\documentclass[10pt,twoside]{siamart1116}
\usepackage{graphicx}
\usepackage[english]{babel}
\usepackage{graphicx,epstopdf,epsfig}
\usepackage{amsfonts,epsfig,fancyhdr,graphics, hyperref,amsmath,amssymb}
\usepackage{enumerate}
\usepackage{xcolor}
\usepackage{tikz}

\newcommand{\qedwhite}{\hfill \ensuremath{\Box}}

\newcommand{\Names}{Moshe Goldberg, Daniel Hershkowitz, and Daniel Szyld}
\newcommand{\Title}{Cospectral graphs with different zero forcing numbers}
\renewcommand{\theequation}{\thesection.\arabic{equation}}
\renewtheorem{theorem}{Theorem}[section]
\newtheorem{remark}[theorem]{Remark}
\newtheorem{example}[theorem]{Example}

\title{Constructions of cospectral graphs with different zero forcing numbers}

\date{}

\begin{document}
\bibliographystyle{plain}
\setcounter{page}{1}
\thispagestyle{empty}
\author{
Aida Abiad
\thanks{\texttt{a.abiad.monge@tue.nl}, Department of Mathematics and Computer Science, Eindhoven University of Technology, The Netherlands\\
Department of Mathematics: Analysis, Logic and Discrete Mathematics, Ghent University, Belgium\\
Department of Mathematics and Data Science, Vrije Universiteit Brussel, Belgium} 
\and 
Boris Brimkov
\thanks{\texttt{boris.brimkov@sru.edu}, Department of Mathematics and Statistics, Slippery Rock University, USA} 
\and 
Jane Breen
\thanks{\texttt{Jane.Breen@ontariotechu.ca}, Faculty of Science, Ontario Tech University, Canada} 
\and
Thomas R. Cameron
\thanks{\texttt{trc5475@psu.edu}, Department of Mathematics, Penn State Behrend, USA} 
\and
Himanshu Gupta 
\thanks{\texttt{himanshu@udel.edu}, Department of Mathematical Sciences, University of Delaware, USA} 
\and 
Ralihe R. Villagr\'an 
\thanks{\texttt{rvillagran@math.cinvestav.mx}, Departamento de Matem\'aticas, Centro de Investigaci\'on y de Estudios Avanzados del IPN, Mexico} 
}
\markboth{\Names}{\Title}

\maketitle

\begin{abstract}
Several researchers have recently explored various graph parameters that can or cannot be characterized by the spectrum of a matrix associated with a graph. In this paper we show that several NP-hard zero forcing numbers are not characterized by the spectra of
several types of associated matrices
 with a graph. In particular, we consider standard zero forcing, positive semidefinite zero forcing, and skew zero forcing, and provide constructions of infinite families of pairs of cospectral graphs which have different values for these numbers. We explore several methods for obtaining these cospectral graphs including using graph products, graph joins, and graph switching. Among these, we provide a construction involving regular adjacency cospectral graphs; the regularity of this construction also implies cospectrality with respect to several other matrices including the Laplacian, signless Laplacian, and normalized Laplacian. We also provide a construction where pairs of cospectral graphs can have an arbitrarily large difference between their zero forcing numbers.

\end{abstract}


\section{Introduction}

It is well-known that the spectrum of a matrix associated with a graph  encodes useful information about the graph. Properties that are
characterized by the spectrum of either the adjacency matrix or the Laplacian matrix of a graph include the number of vertices, the number of edges, and regularity. The spectrum of the adjacency matrix indicates whether or not the graph is bipartite, and the spectrum of the Laplacian matrix indicates whether or not the graph is connected. If a graph is regular (i.e. every vertex has the same degree), the spectrum of the adjacency matrix
can be determined from the spectrum of the Laplacian matrix, and vice versa. This implies that for both matrices the properties of being regular and bipartite, and being regular and connected
are characterized by the spectrum \cite{dh2003}. For other matrices, like the distance matrix, much less is known, see \cite{GRWC2016,grwc,hogbencaroynsurvey}.

If a property is not characterized by the spectrum, then there exists a pair of non-isomorphic graphs with the same spectrum (cospectral graphs), where one graph of the pair has this property, and the other does not. For many graph properties and several types of associated matrices, such pairs of cospectral graphs have been found; such properties include: being distance-regular \cite{h1996}, having a given diameter \cite{hs1995}, admitting a perfect matching \cite{bch2015}, being Hamiltonian \cite{lwyl2010} and having a given vertex or edge-connectivity \cite{h2020}.

Some famous NP-hard graph properties on being characterized by the spectrum
have recently been studied \cite{EH2020}. In this paper we focus on studying whether some other well-known NP-hard \cite{Znphard,Znphard2, Zposnphard} graph properties are characterized by the spectrum: the zero forcing number and some of its variants. The zero forcing number of a graph was
introduced in \cite{zerofircingorigin} as a bound for the minimum rank (or, equivalently, the maximum
nullity) of a matrix whose combinatorial structure is represented by the graph, and it has been extensively studied in the last fifteen years. Not many results relating zero forcing numbers to the spectrum of a graph are known; in \cite{sudakov} a lower bound is given on the zero forcing number of a regular graph in terms of the minimum eigenvalue of the adjacency matrix. 

In this paper we show that the values of several zero forcing numbers do not follow from the spectrum of the most common matrices (adjacency, Laplacian) associated with the graph. To this end, we provide several constructions of cospectral graphs with different zero forcing parameters, using graph products (Section \ref{sec:produtcconstruction}), the join operator (Section \ref{sec:join}) and graph switching (Section \ref{sec:GMconstruction}). As a corollary of our results, it follows that that zero forcing numbers are not determined by the spectrum of the adjacency and Laplacian matrix of a graph.

\section{Preliminaries}
 
Throughout this paper, we consider $G=(V,E)$ to be an undirected simple graph of order $|V|=n$. The \emph{standard zero forcing process} on a graph $G$ is defined as follows. Initially, there is a subset $S$ of blue vertices, while all other vertices are white. The \emph{standard color change rule} dictates that at each step, a blue vertex
with exactly one white neighbor will force its white neighbor to become blue. The set $S$ is said to be a zero forcing set if, by iteratively applying the color change rule, all of $V$ becomes blue. The \emph{zero forcing number} of $G$ is the minimum cardinality of a zero
forcing set in $G$, denoted by $Z(G)$. 

The zero forcing number $Z(G)$ gives a bound on the \emph{maximum nullity} of a graph, defined as follows. Let $G$ be a graph of order $n$, and let $S(G)$ denote the set of $n\times n$ symmetric real matrices $B=[b_{ij}]$ such that $b_{ij}\neq 0$ if and only if $\{i,j\}$ is an edge in $G$; that is, the graph of the matrix $B$, denoted $G(B)$, is equal to $G$. The \emph{maximum nullity} of $G$ and \emph{minimum rank} of $G$ are defined as follows:
\vspace{-0.1cm}
\begin{align*}
   & M (G) = \max \{\text{null}(X):  X\in \mathcal{S}(G)\}\\
   & mr (G) = \min\{\text{rank}(X): X \in \mathcal{S}(G)\}.
\end{align*}
The bound provided by $Z(G)$ is that $M(G) \leq Z(G)$, or equivalently, $mr(G) \geq n-Z(G)$.

A variant of the (standard) zero forcing number $Z(G)$ is the \emph{skew zero forcing number} $Z_{-}(G)$, which uses the same definition with a different color change rule: a vertex $v$ becomes blue if it is the unique white neighbor of some other vertex in the graph (regardless of whether that other vertex is blue or white). If repeated application of the color change rule results in all vertices in the graph eventually being blue, then the initial set of blue vertices is a \emph{skew zero forcing set}. The minimum cardinality of such a set is the skew zero forcing number of $G$. As suggested by the name, the skew zero forcing number provides an upper bound on the maximum nullity $M_{-}$ of skew-symmetric matrices $B$ whose corresponding graph is $G$. More precisely, let $S_{-}(G)$ denote the set of $n\times n$ skew symmetric real matrices $B=[b_{ij}]$ such that $b_{ij}\neq 0$ if and only if $\{i,j\}$ is an edge in $G$; that is, the graph of the matrix $B$, denoted $G(B)$, is equal to $G$. The \emph{maximum skew nullity} of $G$ and \emph{minimum skew rank} of $G$ are defined as follows:
\vspace{-0.1cm}
\begin{align*}
   & M_{-} (G) = \max \{\text{null}(X):  X\in \mathcal{S}_{-}(G)\}\\
   & mr_{-} (G) = \min\{\text{rank}(X): X \in \mathcal{S}_{-}(G)\}.
\end{align*}
The bound provided by $Z_{-}(G)$ is that $M_{-}(G) \leq Z_{-}(G)$, or equivalently, $mr_{-}(G) \geq n-Z_{-}(G)$.

Another variant of the (standard) zero forcing number $Z(G)$ is the \emph{positive semidefinite zero forcing number} $Z_{+}(G)$, which again uses a different color change rule: A vertex $v$ becomes blue if it is a white neighbor of a blue vertex $u$ and no other white neighbor of $u$ is path connected to $v$ via a path of all white vertices.
The minimum cardinality of a set $S$ whose vertices can initially be colored blue and result in the entire graph being colored blue when iteratively applying the positive semidefinite color change rule is referred to as the positive semidefinite zero forcing number of $G$. The positive semidefinite zero forcing number was introduced in \cite{Zplus} as an upper bound for the maximum nullity of a positive semidefinite matrix $B$ whose graph is $G$, which is  denoted by $M_{+}$. That is, let $S_{+}(G)$ denote the set of $n\times n$ positive semidefinite symmetric real matrices $B=[b_{ij}]$ such that $b_{ij}\neq 0$ if and only if $\{i,j\}$ is an edge in $G$; that is, the graph of the matrix $B$, denoted $G(B)$, is equal to $G$.  The \emph{maximum positive semidefinite nullity} of $G$ and the \emph{minimum positive semidefinite rank} of $G$ are defined as follows:
\vspace{-0.1cm}
\begin{align*}
   & M_{+} (G) = \max \{\text{null}(X):  X\in \mathcal{S}_{+}(G)\}\\
   & mr_{+} (G) = \min\{\text{rank}(X): X \in \mathcal{S}_{+}(G)\}.
\end{align*}
The bound provided by $Z_{+}(G)$ is that $M_{+}(G) \leq Z_{+}(G)$, or equivalently, $mr_{+}(G) \geq n-Z_{+}(G)$.
For more background and results on the zero forcing numbers and related parameters, we refer to reader to \cite[Chapters 2 \& 7--10]{bookmrc}.

We denote by $J_{m,n}$ the $m\times n$ all-one matrix, and by $O_{m,n}$ the $m \times n$ all-zero matrix (when size is clear from context, we denote these by $J$ and $O$ respectively). We associate a graph $G$ with a matrix and analyze the eigenvalues of that matrix; these eigenvalues are called the \emph{spectrum} of $G$ with respect to the corresponding associated matrix. Unless otherwise stated, we work with the adjacency matrix and the Laplacian matrix of a graph, which are denoted $A$ and $L$, respectively. Graphs with the same  spectrum are called \emph{cospectral}.

Many of our constructions are supported by numerical experimentation, which use the {\tt Nauty} software~\cite{McKay2013} to identify non-isomorphic cospectral graphs.
For instance,  Figure~\ref{fig:cospectraldifferentZ} provides an example of the smallest pair of cospectral regular graphs with different zero forcing number and positive semidefinite zero forcing number. The graph on the left has zero forcing values of $Z(G)=6$ and $Z_{+}(G)=5$; whereas, the graph on the right has zero forcing values of $Z(G')=4$ and $Z_{+}(G')=4$.
Note that both graphs have a skew zero forcing number equal to $4$.


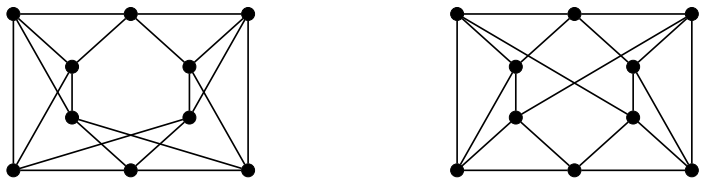
\begin{figure}[ht]
\centering
\begin{tabular}{ccc}
\resizebox{0.2\textwidth}{!}{
\begin{tikzpicture}
	\node[circle,draw=black,fill=black] (1) at (-3,-2) {};
	\node[circle,draw=black,fill=black] (2) at (0,-2) {};
	\node[circle,draw=black,fill=black] (3) at (3,-2) {};
	\node[circle,draw=black,fill=black] (4) at (3,2) {};
	\node[circle,draw=black,fill=black] (5) at (0,2) {};
	\node[circle,draw=black,fill=black] (6) at (-3,2) {};
	
	\node[circle,draw=black,fill=black] (7) at (-1.5,-0.65) {};
	\node[circle,draw=black,fill=black] (8) at (1.5,-0.65) {};
	\node[circle,draw=black,fill=black] (9) at (1.5,0.65) {};
	\node[circle,draw=black,fill=black] (10) at (-1.5,0.65) {};
	
	\foreach \x/\y in {1/2,2/3,3/4,4/5,5/6,6/1}
		\draw[black,=>latex',very thick] (\x) -- (\y);
	\foreach \x/\y in {7/2,7/3,7/6,7/10}
		\draw[black,=>latex',very thick] (\x) -- (\y);
	\foreach \x/\y in {8/1,8/2,8/4,8/9}
		\draw[black,=>latex',very thick] (\x) -- (\y);
	\foreach \x/\y in {9/3,9/4,9/5}
		\draw[black,=>latex',very thick] (\x) -- (\y);
	\foreach \x/\y in {10/1,10/5,10/6}
		\draw[black,=>latex',very thick] (\x) -- (\y);
\end{tikzpicture}%
}
&
\hspace{5em}
&
\resizebox{0.2\textwidth}{!}{
\begin{tikzpicture}
	\node[circle,draw=black,fill=black] (1) at (-3,-2) {};
	\node[circle,draw=black,fill=black] (2) at (0,-2) {};
	\node[circle,draw=black,fill=black] (3) at (3,-2) {};
	\node[circle,draw=black,fill=black] (4) at (3,2) {};
	\node[circle,draw=black,fill=black] (5) at (0,2) {};
	\node[circle,draw=black,fill=black] (6) at (-3,2) {};
	
	\node[circle,draw=black,fill=black] (7) at (-1.5,-0.65) {};
	\node[circle,draw=black,fill=black] (8) at (1.5,-0.65) {};
	\node[circle,draw=black,fill=black] (9) at (1.5,0.65) {};
	\node[circle,draw=black,fill=black] (10) at (-1.5,0.65) {};
	
	\foreach \x/\y in {1/2,2/3,3/4,4/5,5/6,6/1}
		\draw[black,=>latex',very thick] (\x) -- (\y);
	\foreach \x/\y in {7/1,7/2,7/4,7/10}
		\draw[black,=>latex',very thick] (\x) -- (\y);
	\foreach \x/\y in {8/2,8/3,8/6,8/9}
		\draw[black,=>latex',very thick] (\x) -- (\y);
	\foreach \x/\y in {9/3,9/4,9/5}
		\draw[black,=>latex',very thick] (\x) -- (\y);
	\foreach \x/\y in {10/1,10/5,10/6}
		\draw[black,=>latex',very thick] (\x) -- (\y);
\end{tikzpicture}%
}
\end{tabular}
\caption {Smallest pair of connected cospectral regular graphs that have different values for $Z$ and $Z_{+}$ but the same value for $Z_{-}$.}
\label{fig:cospectraldifferentZ}
\end{figure}

A useful method for the construction of non-isomorphic cospectral graphs is described in the following lemma. We use this method in Section~\ref{sec:GMconstruction}.

\begin{lemma}\textup{\cite{GM}[GM switching]}\label{lem:GM}
Let $G$ be a graph and let $\{X_1,\ldots,X_\ell,Y\}$ be a
partition of the vertex set $V(G)$ of $G$.
Suppose that for every vertex $x\in Y$ and every $i\in\{1,\ldots,\ell\}$, $x$ has either $0$,
$\frac{1}{2}|X_i|$ or $|X_i|$ neighbors in $X_i$.
Moreover, suppose that for all $i,j\in\{1,\ldots,\ell\}$ the number of neighbors of an arbitrary vertex of $X_i$ that are contained in $X_j$, depends only on $i$ and $j$ and not on the vertex.
Make a new graph $G'$ from $G$ as follows.
For each $x\in Y$ and $i\in\{1,\ldots,\ell\}$ such that $x$ has $\frac{1}{2}|X_i|$ neighbors
in $X_i$ delete the corresponding $\frac{1}{2}|X_i|$ edges and join $x$ instead to the
$\frac{1}{2}|X_i|$ other vertices in $X_i$.
Then $G$ and $G'$ are cospectral, and their complements are also cospectral.
\end{lemma}

The operation that changes $G$ into $G'$ is called \emph{Godsil-McKay switching}, and the
considered partition is a
\emph{(Godsil-McKay) switching partition}. In many applications of Godsil-McKay switching, $\ell=1$; in that case, the above condition requires that $X=X_1$
 induces a regular subgraph of
$G$, and that each vertex in
$Y$ has $0$, $\frac{1}{2}|X|$ or
 $|X|$ neighbors in $X$. Such a set $X$ will be called a \emph{(Godsil-McKay)  switching  set}. 

\section{Graph products}\label{sec:produtcconstruction}

In this section, we provide several constructions of infinite families of cospectral graphs with different zero forcing numbers, using graph products.

\subsection{Tensor product}\label{subsec:tensor}

The \emph{tensor product} of two graphs $G=(V,E)$ and $G'=(V',E')$ is defined as 
\[
G\times G' = \left(V\times V',\left\{ \left\{\{v_{i},v_{i}'\},\{v_{j},v_{j}'\}\right\}\colon~\{v_{i},v_{j}\}\in E,~\{v_{i}',v_{j}'\}\in E'\right\}\right). 
\]
Note that if $A$ and $A'$ are the adjacency matrices of $G$ and $G'$, respectively, then the adjacency matrix of the tensor product $G\times G'$ is given by $A\otimes A'$, where $\otimes$ denotes the Kronecker product.
In order to show the main result of this section (Proposition \ref{thm:zeroforcingspectrumlesstrivial}), we need to state several preliminary results.

\begin{lemma}\label{prop:tprod-eigvals} {(\cite{bh}, Section 1.4.7)}
Let $A$ be an $n\times n$ matrix and $A'$ be an $n'\times n'$ matrix.
Also, let $ \left\{\lambda_{1},\ldots,\lambda_{n}\right\}$ denote the spectrum of $A$, and let $\left\{\lambda_{1}',\ldots,\lambda_{n'}'\right\}$ denote the spectrum of $G'$.
Then, $\lambda_{i}\lambda_{j}'$ is an eigenvalue of $A\otimes A'$ for $i=1,\ldots,n$ and $j=1,\ldots,n'$.
\end{lemma}

We are now ready to state the main result of this section (Proposition \ref{thm:zeroforcingspectrumlesstrivial}), which requires the existence of cospectral graphs like the ones in the following example in order to produce infinite families of cospectral graphs with different zero forcing numbers.

\begin{example}\label{ex:tprod-cospectral}
Let $G$ and $G'$ be the graphs in Figure~\ref{fig:tprod-cospectral}. One can readily verify that $G$ and $G^{'}$ are cospectral graphs with respect to the adjacency matrix, and $Z_{-}(G) = M_{-}(G) = 3$ and $Z_{-}(G^{'}) = M_{-}(G^{'}) = 1$.
The following Proposition~\ref{thm:zeroforcingspectrumlesstrivial} can be used to construct an infinite family of cospectral graphs that have distinct zero forcing and skew-zero forcing numbers. 

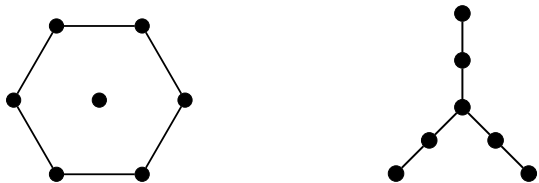
\begin{figure}[ht]
\centering
\begin{tabular}{ccc}
\resizebox{0.15\textwidth}{!}{
\begin{tikzpicture}
	\node[circle,draw=black,fill=black] (1) at (2,0) {};
	\node[circle,draw=black,fill=black] (2) at (1,1.732) {};
	\node[circle,draw=black,fill=black] (3) at (-1,1.732) {};
	\node[circle,draw=black,fill=black] (4) at (-2,0) {};
	\node[circle,draw=black,fill=black] (5) at (-1,-1.732) {};
	\node[circle,draw=black,fill=black] (6) at (1,-1.732) {};
	\node[circle,draw=black,fill=black] (7) at (0,0) {};

	\foreach \x/\y in {1/2,2/3,3/4,4/5,5/6,6/1}
		\draw[black,=>latex',very thick] (\x) -- (\y);
\end{tikzpicture}%
}
&
\hspace{5em}
&
\resizebox{0.12\textwidth}{!}{
\begin{tikzpicture}
	\node[circle,draw=black,fill=black] (1) at (0,0) {};
	\node[circle,draw=black,fill=black] (2) at (0,1) {};
	\node[circle,draw=black,fill=black] (3) at (0,2) {};
	\node[circle,draw=black,fill=black] (4) at (-.707,-.707) {};
	\node[circle,draw=black,fill=black] (5) at (-1.414,-1.414) {};
	\node[circle,draw=black,fill=black] (6) at (.707,-.707) {};
	\node[circle,draw=black,fill=black] (7) at (1.414,-1.414) {};
	
	\foreach \x/\y in {1/2,1/4,1/6,2/3,4/5,6/7}
		\draw[black,=>latex',very thick] (\x) -- (\y);
\end{tikzpicture}%
}
\end{tabular}
\caption {Graphs $G$ and $G'$ from Example \ref{ex:tprod-cospectral}.}
\label{fig:tprod-cospectral}
\end{figure}
\end{example}

\begin{proposition}\label{thm:zeroforcingspectrumlesstrivial}
Let $G_1,G_2$ be two cospectral graphs with respect to the adjacency matrix with order $n$, such that $Z_{-}(G_1)\neq Z_{-} (G_2)$ and $Z_{-}(G_i) = M_{-}(G_i)$ for $i=1,2$. Then there exist infinitely many pairs of cospectral graphs with different zero forcing  and skew-zero forcing numbers.
\end{proposition}

\begin{proof}
Consider the tensor products $G_1\times K_r$ and $G_2\times K_r$.
First, we observe that Lemma \ref{prop:tprod-eigvals} allows us to preserve the cospectrality when considering the tensor products $G_1\times K_r$ and $G_2\times K_r$.\\
Now, by Lemma \ref{thm:tprod-zf-bound} we know that $Z_{-}(G_i\times K_r) = Z(G_i \times K_r) = (r-2)n + 2Z_{-} (G_i)$ for $i=1,2$, and since we assume $Z_{-}(G_1)\neq Z_{-}(G_2)$, it follows that the tensor graph products $G_1\times K_r$ and $G_2\times K_r$ have a different zero forcing number and a different skew zero forcing number.
\end{proof}

For using Proposition \ref{thm:zeroforcingspectrumlesstrivial}, we need to identify cospectral graphs with distinct skew zero forcing numbers that are equal to their respective skew maximum nullity.
The maximum nullity can be computed via the minimum rank by using the fact that the matrix $A$ has rank $r-1$ if and only if all $r$ minors of $A$ are zero and not all $(r-1)$ minors of $A$ are zero~\cite{Brimkov2019}.
Using Python's SymPy library, we implement Algorithm 1 from~\cite{Brimkov2019}.
In particular, we construct a symbolic matrix $A\in S_{-}(G)$ for a given graph $G$.
Then, for $r=1,\ldots,n$, we collect a list of all $r$ minors of $A$ along with the polynomial
\[
1 - \prod_{a_{ij}\neq 0}a_{ij}.
\]
If the Gr{\"o}bner basis of this list of polynomials does not contain the polynomial $1$, then the matrix $A$ has rank $r-1$.
The code for the minimum rank calculation and the rest of our numerical experiments is available on \href{https://github.com/trcameron/CospectralZF}{GitHub} \cite{github}.

\subsection{Cartesian product}\label{sec:positivesemidefinite}

The \emph{Cartesian product} of two graphs $G$ and $G'$ is denoted by $G \square G'$ and is defined as follows: a vertex $(g,g')$ is adjacent to a vertex $(h,h')$ in $G\square H$ if and only if 
\begin{enumerate}
    \item $g=h$ and $\{g',h'\} \in E(G')$, or
    \item $g' = h'$ and $\{g,h\} \in E(G)$.
\end{enumerate}

The following lemma provides the exact expression of the positive semidefinite zero forcing number and maximum positive semidefinite nullity of $K_r \square K_r$ in terms of $r$. Recall that the \emph{line graph} $L(G)$ of $G$ is the graph with the edge set of $G$ as vertex set, where
two vertices are adjacent if the corresponding edges of $G$ have an endpoint in
common.

\begin{lemma}\label{ZplusofKrsqaure}
Let $r \geq 2$. Then $Z_{+}(K_r \square  K_r) = M_{+}(K_r \square  K_r) = (r-1)^2+1$.
\end{lemma}
\begin{proof}

Let us denote the line graph of the complete bipartite graph $K_{r,r}$ by $L(K_{r,r})$. Note that $L(K_{r,r}) = K_r \square K_r$. By using \cite[Theorem 4.5]{Linegraph}, we have
\begin{align*}
Z_+(L(K_{r,r})) = M_+(L(K_{r,r})) &= |E(K_{r,r})| - |V(K_{r,r})| + 2\\
&= r^2 - 2r + 2 = (r-1)^2 + 1.
\end{align*}
\end{proof}

\begin{proposition}\label{propo:grdshrikande}
Let $r\geq 11$. Let $G_1$ be the $4\times 4$ grid (left of Figure~\ref{fig:grdshrikande}), and let $G_2$ be the graph obtained by switching $G_1$ with respect to a 4-coclique (right of Figure \ref{fig:grdshrikande}). Then  $Z_+(G_1 \square K_r) \neq Z_+(G_2 \square K_r)$. 
\end{proposition}
\begin{proof}
First, observe that $G_1$ ($4\times 4$ grid or lattice graph $L_2(4)$) and $G_2$ (Shrikhande graph) are two  cospectral and non-isomorphic graphs \cite{bh}. By Lemma \ref{ZplusofKrsqaure}, we have that $Z_+(G_1) = M_+(G_1) = 10$, and one  can easily check by SageMath that $Z_+(G_2) = 9$. It is known that $Z_{+}(K_r) = M_+(K_r) = r-1$ \cite[Theorem 9.47-(2)]{bookmrc}.

By using \cite[Proposition 9.57]{bookmrc}, we have that

\begin{align*}
Z_+(G_2 \square K_r) \leq \min\{rZ_{+}(G_2), 16Z_{+}(K_r)\} = \min\{9r,16r-16\} = 9r.
\end{align*}
By using \cite[Corollary 9.60]{bookmrc}, we have that
\begin{align*}
Z_{+}(G_1 \square K_r) \geq Z_{+}(G_1)Z_+(K_r) = 10(r-1). 
\end{align*}
Thus, since $r\geq 11$, it holds that $Z_+(G_2 \square K_r) < Z_{+}(G_1 \square K_r)$.
\end{proof}

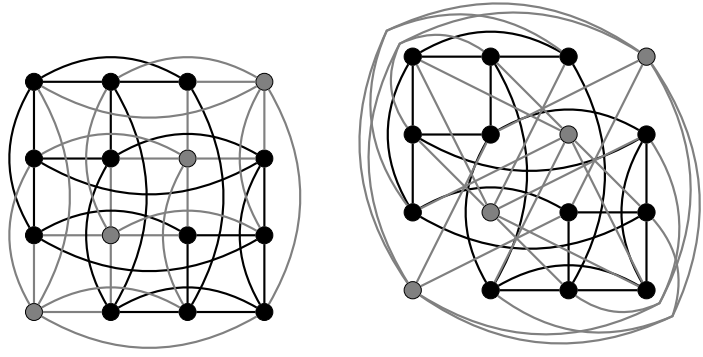
\begin{figure}
\centering
\begin{tabular}{cc}
\resizebox{0.25\textwidth}{!}{
\begin{tikzpicture} 
	\node[circle,draw=black,fill=black] (1) at (-3,3) {};
	\node[circle,draw=black,fill=black] (2) at (-1.5,3) {};
	\node[circle,draw=black,fill=black] (3) at (0,3) {};
	\node[circle,draw=black,fill=gray] (4) at (1.5,3) {};
	
	\node[circle,draw=black,fill=black] (5) at (-3,1.5) {};
	\node[circle,draw=black,fill=black] (6) at (-1.5,1.5) {};
	\node[circle,draw=black,fill=gray] (7) at (0,1.5) {};
	\node[circle,draw=black,fill=black] (8) at (1.5,1.5) {};
	
	\node[circle,draw=black,fill=black] (9) at (-3,0) {};
	\node[circle,draw=black,fill=gray] (10) at (-1.5,0) {};
	\node[circle,draw=black,fill=black] (11) at (0,0) {};
	\node[circle,draw=black,fill=black] (12) at (1.5,0) {};
	
	\node[circle,draw=black,fill=gray] (13) at (-3,-1.5) {};
	\node[circle,draw=black,fill=black] (14) at (-1.5,-1.5) {};
	\node[circle,draw=black,fill=black] (15) at (0,-1.5) {};
	\node[circle,draw=black,fill=black] (16) at (1.5,-1.5) {};
	
	\foreach \x/\y in {1/2,2/3,1/5,2/6,5/6,5/9,8/12,11/12,11/15,12/16,14/15,15/16}
		\draw[black,=>latex',very thick] (\x) -- (\y);
	\foreach \x/\y in {3/4,3/7,4/8,6/7,6/10,7/8,7/11,9/10,9/13,10/11,10/14,13/14}
		\draw[gray,=>latex',very thick] (\x) -- (\y);
	\foreach \x/\y in {1/3,9/1,3/15,2/14,8/5,6/8,14/6,16/8,12/9,9/11,14/16}
		\draw[black,=>latex',very thick] (\x) to [bend left,looseness=1] (\y);
	\foreach \x/\y in {4/1,2/4,1/13,10/2,12/4,4/16,5/7,13/5,15/7,10/12,13/15,16/13}
		\draw[gray,=>latex',very thick] (\x) to [bend left,looseness=1] (\y);
\end{tikzpicture}%
}
&
\resizebox{0.3\textwidth}{!}{
\begin{tikzpicture}
	\node[circle,draw=black,fill=black] (1) at (-3,3) {};
	\node[circle,draw=black,fill=black] (2) at (-1.5,3) {};
	\node[circle,draw=black,fill=black] (3) at (0,3) {};
	\node[circle,draw=black,fill=gray] (4) at (1.5,3) {};
	
	\node[circle,draw=black,fill=black] (5) at (-3,1.5) {};
	\node[circle,draw=black,fill=black] (6) at (-1.5,1.5) {};
	\node[circle,draw=black,fill=gray] (7) at (0,1.5) {};
	\node[circle,draw=black,fill=black] (8) at (1.5,1.5) {};
	
	\node[circle,draw=black,fill=black] (9) at (-3,0) {};
	\node[circle,draw=black,fill=gray] (10) at (-1.5,0) {};
	\node[circle,draw=black,fill=black] (11) at (0,0) {};
	\node[circle,draw=black,fill=black] (12) at (1.5,0) {};
	
	\node[circle,draw=black,fill=gray] (13) at (-3,-1.5) {};
	\node[circle,draw=black,fill=black] (14) at (-1.5,-1.5) {};
	\node[circle,draw=black,fill=black] (15) at (0,-1.5) {};
	\node[circle,draw=black,fill=black] (16) at (1.5,-1.5) {};
	
	\foreach \x/\y in {1/2,2/3,1/5,2/6,5/6,5/9,8/12,11/12,11/15,12/16,14/15,15/16}
		\draw[black,=>latex',very thick] (\x) -- (\y);
	\foreach \x/\y in {1/3,9/1,3/15,2/14,8/5,6/8,14/6,16/8,12/9,9/11,14/16}
		\draw[black,=>latex',very thick] (\x) to [bend left,looseness=1] (\y);
		
	\foreach \x/\y in {4/6,4/11,1/7,2/7,7/12,7/16,13/6,13/11,15/10,16/10,10/5,10/1,10/3,10/8,7/9,7/14}
		\draw[gray,=>latex',very thick] (\x) -- (\y);
		
	\draw[gray,=>latex',very thick,bend left,looseness=1] (4) to (2.0,-2.0) to (14);
	\draw[gray,=>latex',very thick,bend left,looseness=1] (4) to (1.75,-1.75) to (15);
	\draw[gray,=>latex',very thick,bend right,looseness=1] (13) to (1.75,-1.75) to (8);
	\draw[gray,=>latex',very thick,bend right,looseness=1] (13) to (2.0,-2.0) to (12);
	
	\draw[gray,=>latex',very thick,bend right,looseness=1] (4) to (-3.25,3.25) to (5);
	\draw[gray,=>latex',very thick,bend right,looseness=1] (4) to (-3.5,3.5) to (9);
	\draw[gray,=>latex',very thick,bend left,looseness=1] (13) to (-3.25,3.25) to (2);
	\draw[gray,=>latex',very thick,bend left,looseness=1] (13) to (-3.5,3.5) to (3);
\end{tikzpicture}%
}
\end{tabular}
\caption{The graph $G_{1}$ (left) and $G_{2}$ (right) from Proposition~\ref{propo:grdshrikande}}
\label{fig:grdshrikande}
\end{figure}

\section{Graph join }\label{sec:join}

The \emph{join} of two graphs $G$ and $H$, denoted by $G \vee H$, is their disjoint union with all edges of the form $\{u,v\}$ added, where $u\in V (G)$ and $v\in V (H)$. In this section we provide  constructions of cospectral graphs with different zero forcing numbers using graph joins. We begin with a preliminary result regarding the Laplacian spectrum of the join of two graphs. 

\begin{lemma}\label{lemma:fiedler-lap}\cite[Corollary 3.7]{m1991}
Let $L$ and $L'$ denote the Laplacian matrices of two graphs $G$ and $G'$, respectively.
Also, let $\lambda_{1}\leq\lambda_{2}\leq\cdots\leq\lambda_{n}$ denote the eigenvalues of $L$ and $\lambda_{1}'\leq\lambda_{2}'\leq\cdots\leq\lambda_{n'}'$ denote the eigenvalues of $L'$.
Then, the Laplacian matrix of $G\vee G'$ has eigenvalues $\lambda_{2}+n',\ldots,\lambda_{n}+n',\lambda_{2}'+n,\ldots,\lambda_{n'}'+n,n+n',0$.
\end{lemma}

Note that Lemma \ref{lemma:fiedler-lap} also provides the spectrum of the adjacency matrix associated with the join of two regular graphs:

\begin{remark}\label{cor:fiedler-adj}
Let $G$ and $G'$ be regular graphs of degree $r$ and $r'$, respectively. 
Also, let $A$ and $A'$ denote the adjacency matrices of $G$ and $G'$, respectively, with corresponding spectrum $\sigma(A) = \left\{\lambda_{1},\ldots,\lambda_{n}\right\}$ and $\sigma(A') = \left\{\lambda_{1}',\ldots,\lambda_{n'}'\right\}$.
Then, the adjacency matrix of $G\vee G'$ has eigenvalues $$\lambda_{2},\ldots,\lambda_{n},\lambda_{2}',\ldots,\lambda_{n'}',\gamma_{1},\gamma_{2},$$ where $\gamma_{1},\gamma_{2}$ are the roots of the polynomial $\lambda^{2}-\lambda(r+r')+(rr'-nn')$.
\end{remark}

The zero forcing number of the join of two graphs based on the zero forcing number of each of the initial graphs and their orders is shown in \cite[Theorem 5.3.1]{fallatstudentphd}, and stated below. For completeness, we include its proof, and  we show that this result can also be extended to the skew zero forcing number. 
 
\begin{lemma}\label{prop:zf-join}
Let $G=(V,E)$ and $G'=(V',E')$ be connected graphs.
Then, the zero forcing number of the join $G\vee G'$ satisfies 
\[
Z(G\vee G') = \min\left\{n + Z(G'),n'+Z(G)\right\}.
\]
Moreover, the skew-zero forcing number satisfies
\[
Z_{-}(G\vee G') = \min\left\{n + Z_{-}(G'),n'+Z_{-}(G)\right\}.
\]
\end{lemma}
\begin{proof}
Suppose that at a certain step of the zero forcing process, the vertex $u$ forces the vertex $v$, where $v\in V$.
Note that if $u\in V'$, then all vertices in $V\setminus{\{v\}}$ must be colored blue; otherwise, $u$ would have more than one white neighbor and therefore would be unable to force any vertex.
However, since $G$ is connected, it follows that the force may be performed by a different vertex in $V$. 
Hence, we may assume that the forcing is done by a vertex $u\in V$, in which case it follows that all vertices in $V'$ must be colored blue. 
Therefore, a minimum zero forcing set is obtained by coloring all vertices in $V'$ and selecting a minimum zero forcing set of $G$.
Similarly, if the vertex $u$ forces the vertex $v$,  where $v\in V'$,
then a minimum zero forcing set is obtained by coloring all vertices in $V$ and selecting a minimum zero forcing set of $G'$. 

Since the above arguments do not depend on the state of the vertex $u$ performing the forcing (blue or white), it follows that this result holds for both the standard zero forcing rule and the skew forcing rule. 
\end{proof}

Finally, we use the previous results to construct a family of cospectral graphs with distinct zero forcing numbers.

\begin{proposition}\label{prop:cospectral-lap}
Let $G_{1},G_{2}$ denote two connected graphs that are cospectral with respect to the Laplacian, and that have distinct zero forcing (skew zero forcing) numbers.
Then, for $r\geq 1$, the joins $G_{1}\vee K_{r}$ and $G_{2}\vee K_{r}$ are Laplacian-cospectral and have distinct zero forcing (skew zero forcing) numbers. 
\end{proposition}
\begin{proof}
Since $G_{1}$ and $G_{2}$ have equal Laplacian spectrum, it follows from Lemma~\ref{lemma:fiedler-lap} that the joins $G_{1}\vee K_{r}$ and $G_{2}\vee K_{r}$ have equal Laplacian spectrum. 
Furthermore, since the complete graph is determined by its Laplacian spectrum, it follows that $G_{1}$ and $G_{2}$ are not the complete graph of order $n$.
Hence, $\max\left\{Z(G_{1}),Z(G_{2})\right\} < (n-1)$ and $\max\left\{Z_{-}(G_{1}),Z_{-}(G_{2})\right\} < (n-2)$, and it follows from Proposition~\ref{prop:zf-join} that the zero forcing (skew-zero forcing) number of the joins satisfy
\begin{align*}
Z(G_{1}\vee K_{r}) &= r+Z(G_{1}), \\
Z(G_{2}\vee K_{r}) &= r+Z(G_{2}).
\end{align*}
Hence, the joins have distinct zero forcing (skew-zero forcing) numbers.
\end{proof}

\begin{remark}\label{prop:cospectral-adj}
If $G_1$ and $G_2$ are both $k$-regular graphs of the same order that are cospectral with respect to the adjacency matrix, then Proposition \ref{prop:cospectral-lap} also applies to create regular adjacency-cospectral graphs with distinct zero forcing and skew zero forcing numbers. Note that, for instance, there exist many cospectral regular graphs among strongly regular graphs \cite{srgbook}.
\end{remark}

If $G$ is a connected graph, we define $\vee_0 G=G$ and define $\vee_{k+1} G=(\vee_k G)\vee G$ for $k\geq 0$.
\begin{theorem}
Let $G$ and $H$ be two connected graphs of order $n$ with equal Laplacian spectrum and distinct zero forcing (skew-zero forcing) numbers.
Then, for $k\geq 1$, $\vee_k G$ and $\vee_k H$ have equal Laplacian spectrum and distinct zero forcing (skew zero forcing) numbers.
\end{theorem}
\begin{proof}
We proceed by induction. By Lemma~\ref{lemma:fiedler-lap}, $\vee_1 G$ and $\vee_1 H$ have equal Laplacian spectrum and by Lemma \ref{prop:zf-join} we have that
\[Z(\vee_1 G)=n+Z(G)\neq n+Z(H)=Z(\vee_1 H) .  \]  
Now, let $k\geq 2$ and assume the result holds for every integer $j\in \{1,2,\ldots,k\}$ and that, in particular, $Z(\vee_j G)=jn+Z(G)\neq jn+Z(H)=Z(\vee_j H)$. Since $\vee_k G$ and $\vee_k H$ have equal Laplacian spectrum and $G$ and $H$ have equal Laplacian spectrum, Lemma~\ref{lemma:fiedler-lap} implies that $\vee_{k+1} G$ and $\vee_{k+1} H$ have equal Laplacian spectrum.
Furthermore, by Lemma \ref{prop:zf-join},
\[
Z(\vee_{k+1} G)=(k+1)n+Z(G)\neq (k+1)n+Z(H)=Z(\vee_{k+1} H).
\]
An analogous argument applies for the skew zero forcing numbers.
\end{proof}

\section{Graph switching}\label{sec:GMconstruction}

In this section, we provide constructions of cospectral graphs with different zero forcing numbers using Godsil-McKay switching (see Lemma \ref{lem:GM}). We begin with a construction where pairs of cospectral graphs can have an arbitrarily large difference between their zero forcing numbers.

\begin{theorem}
\label{thm1}
Let $G_1$ and $G_2$ be connected $k$-regular graphs with $|V(G_1)|=|V(G_2)|=:n$ and $Z(G_1)\neq Z(G_2)$. For some $m\geq n$, let $G'=(G_1\vee P_m)\cup G_2$ and $G''=(G_2\vee P_m)\cup G_1$, where $\cup$ denotes the disjoint union. Then, $G'$ and $G''$ are nonisomorphic cospectral graphs with $Z(G')\neq Z(G'')$.
\end{theorem}

\proof
Let $X=V(G_1)\cup V(G_2)$. The vertices in $X$ induce a regular subgraph of $G'$, and each vertex in $V(G')\backslash X$ is adjacent to exactly half of the vertices in $X$. Thus, $X$ is a Godsil-McKay switching set of $G'$, and by construction, applying Godsil-McKay switching yields precisely $G''$; thus, $G'$ and $G''$ are cospectral.

Next we show that $Z(G')\neq Z(G'')$ and therefore $G'$ and $G''$ are cospectral but nonisomorphic. For $i\in \{1,2\}$, let $S_i$ be a minimum zero forcing set of $G_i$.
Let $v$ be an endpoint of the path $P_m$ used in the construction of $G'$ and $G''$. Let $S'=V(G_1)\cup S_2\cup \{v\}$ and $S''=S_1\cup V(G_2)\cup \{v\}$. $S'$ is a zero forcing set of $G'$, since $S_2$ can force $G_2$ and $v$ can force $P_m$ (since each neighbor of a vertex in $P_m$ that is not itself in $P_m$ is in $S'$). Similarly, $S''$ is a zero forcing set of $G''$, since $S_1$ can force $G_1$ and $v$ can force $P_m$. Now, suppose for contradiction that $Z(G')<|S'|$. Since $S_2$ is a minimum zero forcing set of the component $G_2$ of $G'$, it follows that $Z(G_1\vee P_m)<|S'|-|S_2|=|V(G_1)|+1=n+1$. However, this contradicts Lemma \ref{prop:zf-join}, since 
\begin{eqnarray*}
Z(G_1\vee P_m) &=&\min\left\{|V(G_1)| + Z(P_m),|V(P_m)|+Z(G_1)\right\}\\
&=&\min\left\{n + 1,m+S_1\right\}=n+1.
\end{eqnarray*}
Thus, $Z(G')=|S'|$, and similarly $Z(G'')=|S''|$. Moreover, $|S'|=|V(G_1)|+|S_2|+1\neq |V(G_2)|+|S_1|+1=|S''|$, since $|V(G_1)|=|V(G_2)|$ but $Z(G_1)\neq Z(G_2)$. Finally, since $Z(G')\neq Z(G'')$, $G'$ and $G''$ are nonisomorphic.
\qed

\begin{corollary}
For any positive integer $k$, there exist nonisomorphic cospectral graphs $G'$ and $G''$ with $Z(G')>Z(G'')+k$. 
\end{corollary}
\proof
Let $c$ be an integer greater than $k/4+2$ and let $G_1=C_4\square C_{c^2}$ and $G_2=C_{2c}\square C_{2c}$. It was shown in \cite[Theorem 2.6]{Benson2018} that for $3\leq s\leq t$, 
\[Z(C_s\square C_t)=\begin{cases} 2s-1 &\mbox{if } s\text{ and }t \text{ are equal and odd}\\
2s& \mbox{otherwise.} \end{cases}
\]
Thus, $Z(G_1)=8$ and $Z(G_2)=4c$. Moreover, $G_1$ and $G_2$ are 4-regular graphs with $4c^2$ vertices, so $G_1$ and $G_2$ satisfy the conditions of Theorem \ref{thm1}. By the last paragraph of the proof of Theorem \ref{thm1}, for graphs $G'$ and $G''$ constructed as in the statement of Theorem \ref{thm1}, 
\begin{eqnarray*}
Z(G')-Z(G'')&=&(|V(G_1)|+Z(G_2)+1)-(|V(G_2)|+Z(G_1)+1)\\
&=&Z(G_2)-Z(G_1)=4c-8>k.
\end{eqnarray*}
\qed

Note that in Proposition \ref{thm:zeroforcingspectrumlesstrivial} we showed an alternative construction which uses graph products to obtain many cospectral graphs having different $Z$. 

\subsection{A regular construction}


The regularity requirement makes the construction of cospectral graphs more complicated, but also a lot more interesting. The reason is that cospectrality for regular graphs with respect to the adjacency matrix implies cospectrality with respect to several other kinds of matrices, such as the Laplacian matrix, the signless Laplacian and the normalized Laplacian. GM switching constructions of cospectral graphs have been used in the literature to show that several graph parameters do not follow from the spectrum of $A$ and $L$, see \cite{bch2015,h2020}. In fact, we will show that the construction of regular cospectral graphs with different vertex connectivity from \cite[Section 2]{h2020} provides a pair of families of cospectral graphs with different zero forcing numbers. First we describe the construction, and later on we focus on showing that they have different zero forcing numbers. 

Let $k\geq 2$. We define a $k$-regular graph $H$ with vertex set $V(H) = \{0, 1, 2, \ldots, 3k-2\}$ as follows. For $i = 0,\ldots, 3k-2$ vertex $i$ is adjacent to $\{k+i,k+i+1,\ldots,2k+i-1\}$ (mod $3k-1$). 
Now let $A$ be a clique of order $k+1$, let $B$ be a coclique of order $k+1$, and let $C$ be a coclique of order $k-1$, with vertex sets
$V(A)= \{a_0,\ldots ,a_k\}$, $V(B) = \{b_0,\ldots,b_k\}$, and $V(C)= \{c_1,\ldots,c_{k-1}\}$. We construct a $2k$-regular graph $G$ of order $6k$ by taking a disjoint union of $H$, $A$, $B$, and $C$, and adding edges as follows.
\begin{enumerate}[(i)]
\item For all $0\leq i,j \leq k$ and $j\neq i$, $a_i \sim b_j$. 
\item For all $1\leq i \leq k-1$ and $j\in \{0,k,\ldots,3k-2\}$, $c_i \sim j$.
\item For all $0\leq i \leq k-1$ and $j \in \{1,\ldots,k-1\}$, $b_i \sim j$. For all $j \in\{2k,\ldots,3k-2\}$, $b_k \sim j$.
\item For all $1\leq i \leq k$, $b_i \sim k+i-1$ and $b_0 \sim 0$. 
\end{enumerate}
We illustrate the construction of $G$ in Figure \ref{fig:haemersconstruction}. The solid arrows between subgraphs denote that all possible edges between vertices in those subgraphs are present; the hollow arrow represents all possible edges minus the perfect matching as in (i) above; and the curved solid arrow denotes the perfect matching as described in (iv) above. We denote the clique of order $n$ by $K_n$, and a coclique of order $n$ by $O_n$. 

\begin{figure}[ht]
    \centering
\tikzset{every picture/.style={line width=0.75pt}} 

\begin{tikzpicture}[x=0.58pt,y=0.5pt,yscale=-1,xscale=1]

\draw   (307,176.43) .. controls (307.02,162.34) and (334.57,150.97) .. (368.53,151.01) .. controls (402.5,151.06) and (430.02,162.52) .. (430,176.6) .. controls (429.98,190.68) and (402.43,202.06) .. (368.46,202.01) .. controls (334.5,201.97) and (306.98,190.51) .. (307,176.43) -- cycle ;
\draw   (301,68.5) .. controls (301,53.31) and (323.16,41) .. (350.5,41) .. controls (377.84,41) and (400,53.31) .. (400,68.5) .. controls (400,83.69) and (377.84,96) .. (350.5,96) .. controls (323.16,96) and (301,83.69) .. (301,68.5) -- cycle ;
\draw   (155,166.75) .. controls (155,154.46) and (180.52,144.5) .. (212,144.5) .. controls (243.48,144.5) and (269,154.46) .. (269,166.75) .. controls (269,179.04) and (243.48,189) .. (212,189) .. controls (180.52,189) and (155,179.04) .. (155,166.75) -- cycle ;
\draw   (469,121) .. controls (469,91.18) and (499,67) .. (536,67) .. controls (573,67) and (603,91.18) .. (603,121) .. controls (603,150.82) and (573,175) .. (536,175) .. controls (499,175) and (469,150.82) .. (469,121) -- cycle ;
\draw   (16,166.75) .. controls (16,153.36) and (42.19,142.5) .. (74.5,142.5) .. controls (106.81,142.5) and (133,153.36) .. (133,166.75) .. controls (133,180.14) and (106.81,191) .. (74.5,191) .. controls (42.19,191) and (16,180.14) .. (16,166.75) -- cycle ;
\draw   (67,85) .. controls (67,76.44) and (104.38,69.5) .. (150.5,69.5) .. controls (196.62,69.5) and (234,76.44) .. (234,85) .. controls (234,93.56) and (196.62,100.5) .. (150.5,100.5) .. controls (104.38,100.5) and (67,93.56) .. (67,85) -- cycle ;
\draw   (50.9,75) .. controls (49.24,60.09) and (94.82,48) .. (152.7,48) .. controls (210.58,48) and (258.84,60.09) .. (260.5,75) .. controls (262.16,89.91) and (216.58,102) .. (158.7,102) .. controls (100.82,102) and (52.55,89.91) .. (50.9,75) -- cycle ;
\draw   (301,68.5) .. controls (301,48.34) and (329.43,32) .. (364.5,32) .. controls (399.57,32) and (428,48.34) .. (428,68.5) .. controls (428,88.66) and (399.57,105) .. (364.5,105) .. controls (329.43,105) and (301,88.66) .. (301,68.5) -- cycle ;
\draw    (136,62.5) -- (90,79.5) ;
\draw    (136,62.5) -- (185,79.5) ;
\draw  [color={rgb, 255:red, 0; green, 0; blue, 0 }  ][line width=0.75] [line join = round][line cap = round] (99,84.5) .. controls (99,84.5) and (99,84.5) .. (99,84.5) ;
\draw    (90,79.5) -- (173,168) ;
\draw    (90,79.5) -- (247,167.5) ;
\draw    (185,79.5) -- (29,163.5) ;
\draw    (185,79.5) -- (114,162.5) ;
\draw    (114,173.5) .. controls (137,192.5) and (124,194.5) .. (175,178.5) ;
\draw    (114,173.5) .. controls (84,227.5) and (126,291.5) .. (243,175.5) ;
\draw    (29,163.5) .. controls (25,255.5) and (135,208.5) .. (175,178.5) ;
\draw    (29,163.5) -- (107,85.5) ;
\draw    (29,163.5) -- (64,136.5) ;
\draw    (90,79.5) -- (124,106.5) ;
\draw    (185,79.5) -- (152,105.5) ;
\draw    (175,178.5) -- (138,194.5) ;
\draw    (170,80.5) -- (247,167.5) ;
\draw    (214,138.5) -- (247,167.5) ;
\draw    (41,167.5) ;
\draw [line width=1.5]    (412,77.5) -- (267.64,156.07) ;
\draw [shift={(265,157.5)}, rotate = 331.44] [color={rgb, 255:red, 0; green, 0; blue, 0 }  ][line width=1.5]    (14.21,-4.28) .. controls (9.04,-1.82) and (4.3,-0.39) .. (0,0) .. controls (4.3,0.39) and (9.04,1.82) .. (14.21,4.28)   ;
\draw [line width=1.5]    (265,157.5) -- (409.36,78.93) ;
\draw [shift={(412,77.5)}, rotate = 511.44] [color={rgb, 255:red, 0; green, 0; blue, 0 }  ][line width=1.5]    (14.21,-4.28) .. controls (9.04,-1.82) and (4.3,-0.39) .. (0,0) .. controls (4.3,0.39) and (9.04,1.82) .. (14.21,4.28)   ;
\draw [line width=1.5]    (258,80.5) -- (325.95,153.31) ;
\draw [shift={(328,155.5)}, rotate = 226.97] [color={rgb, 255:red, 0; green, 0; blue, 0 }  ][line width=1.5]    (14.21,-4.28) .. controls (9.04,-1.82) and (4.3,-0.39) .. (0,0) .. controls (4.3,0.39) and (9.04,1.82) .. (14.21,4.28)   ;
\draw [line width=1.5]    (328,155.5) -- (260.05,82.69) ;
\draw [shift={(258,80.5)}, rotate = 406.97] [color={rgb, 255:red, 0; green, 0; blue, 0 }  ][line width=1.5]    (14.21,-4.28) .. controls (9.04,-1.82) and (4.3,-0.39) .. (0,0) .. controls (4.3,0.39) and (9.04,1.82) .. (14.21,4.28)   ;
\draw [line width=1.5]    (269,166.75) -- (311.35,189.1) ;
\draw [shift={(314,190.5)}, rotate = 207.82] [color={rgb, 255:red, 0; green, 0; blue, 0 }  ][line width=1.5]    (14.21,-4.28) .. controls (9.04,-1.82) and (4.3,-0.39) .. (0,0) .. controls (4.3,0.39) and (9.04,1.82) .. (14.21,4.28)   ;
\draw [line width=1.5]    (314,190.5) -- (271.65,168.15) ;
\draw [shift={(269,166.75)}, rotate = 387.82] [color={rgb, 255:red, 0; green, 0; blue, 0 }  ][line width=1.5]    (14.21,-4.28) .. controls (9.04,-1.82) and (4.3,-0.39) .. (0,0) .. controls (4.3,0.39) and (9.04,1.82) .. (14.21,4.28)   ;
\draw [line width=1.5]    (333,94.5) -- (110.92,145.82) ;
\draw [shift={(108,146.5)}, rotate = 346.99] [color={rgb, 255:red, 0; green, 0; blue, 0 }  ][line width=1.5]    (14.21,-4.28) .. controls (9.04,-1.82) and (4.3,-0.39) .. (0,0) .. controls (4.3,0.39) and (9.04,1.82) .. (14.21,4.28)   ;
\draw [line width=1.5]    (108,146.5) -- (330.08,95.18) ;
\draw [shift={(333,94.5)}, rotate = 526.99] [color={rgb, 255:red, 0; green, 0; blue, 0 }  ][line width=1.5]    (14.21,-4.28) .. controls (9.04,-1.82) and (4.3,-0.39) .. (0,0) .. controls (4.3,0.39) and (9.04,1.82) .. (14.21,4.28)   ;
\draw   (427.27,61.82) -- (444.76,60) -- (443.39,63.66) -- (472.91,74.66) -- (474.27,71) -- (486.3,83.83) -- (468.82,85.65) -- (470.18,81.99) -- (440.66,70.99) -- (439.3,74.65) -- cycle ;
\draw [line width=1.5]    (213,46) .. controls (239.6,13) and (239.99,53.26) .. (304.03,53.52) ;
\draw [shift={(307,53.5)}, rotate = 539.14] [color={rgb, 255:red, 0; green, 0; blue, 0 }  ][line width=1.5]    (14.21,-4.28) .. controls (9.04,-1.82) and (4.3,-0.39) .. (0,0) .. controls (4.3,0.39) and (9.04,1.82) .. (14.21,4.28)   ;
\draw [line width=1.5]    (307,53.5) .. controls (244.94,59.41) and (240.13,10.5) .. (210.38,48.69) ;
\draw [shift={(209,50.5)}, rotate = 306.76] [color={rgb, 255:red, 0; green, 0; blue, 0 }  ][line width=1.5]    (14.21,-4.28) .. controls (9.04,-1.82) and (4.3,-0.39) .. (0,0) .. controls (4.3,0.39) and (9.04,1.82) .. (14.21,4.28)   ;
\draw    (114,173.5) -- (114,194.5) ;

\draw (129,49) node [anchor=north west][inner sep=0.75pt]    {\scriptsize{$0$}};
\draw (79,77) node [anchor=north west][inner sep=0.75pt]    {\scriptsize{$k$}};
\draw (184,77) node [anchor=north west][inner sep=0.25pt]    {\scriptsize{$2k-1$}};
\draw (18,162) node [anchor=north west][inner sep=0.75pt]    {\scriptsize{$1$}};
\draw (83,162) node [anchor=north west][inner sep=0.75pt]   {\scriptsize{$k-1$}};
\draw (158,162) node [anchor=north west][inner sep=0.75pt]    {\scriptsize{$2k$}};
\draw (211,162) node [anchor=north west][inner sep=0.75pt]    {\scriptsize{$3k-2$}};
\draw (350,110) node [anchor=north west][inner sep=0.75pt]    {$O_{k}{}_{+}{}_{1}$};
\draw (520,182) node [anchor=north west][inner sep=0.75pt]    {$K_{k}{}_{+}{}_{1}$};
\draw (352,204) node [anchor=north west][inner sep=0.75pt]    {$O_{k}{}_{- 1}$};
\draw (135,240) node [anchor=north west][inner sep=0.75pt]    {$H$};
\draw (529,80) node [anchor=north west][inner sep=0.75pt]    {\scriptsize{$a_{0}$}};
\draw (491,127) node [anchor=north west][inner sep=0.75pt]    {\scriptsize{$a_{1}$}};
\draw (562,127) node [anchor=north west][inner sep=0.75pt]    {\scriptsize{$a_{k}$}};
\draw (303,57) node [anchor=north west][inner sep=0.75pt]    {\scriptsize{$b_{0}$}};
\draw (358,57) node [anchor=north west][inner sep=0.75pt]    {\scriptsize{$b_{k}{}_{- 1}$}};
\draw (404,57) node [anchor=north west][inner sep=0.75pt]    {\scriptsize{$b_{k}$}};
\draw (320,164) node [anchor=north west][inner sep=0.75pt]    {\scriptsize{$c_{1}$}};
\draw (380,164) node [anchor=north west][inner sep=0.75pt]    {\scriptsize{$c_{k}{}_{- 1}$}};
\draw (122,80) node [anchor=north west][inner sep=0.75pt]    {$\dotsc $};
\draw (48,165) node [anchor=north west][inner sep=0.75pt]    {$\dotsc $};
\draw (184,165) node [anchor=north west][inner sep=0.75pt]    {$\dotsc $};
\draw (346,168) node [anchor=north west][inner sep=0.75pt]    {$\dotsc $};
\draw (328,62) node [anchor=north west][inner sep=0.75pt]    {$\dotsc $};
\draw (519,130) node [anchor=north west][inner sep=0.75pt]    {$\dotsc $};
\end{tikzpicture}
\caption{The construction of the graph $G$ used to prove Theorem \ref{thm:GMregularcase}} \label{fig:haemersconstruction}
\end{figure}
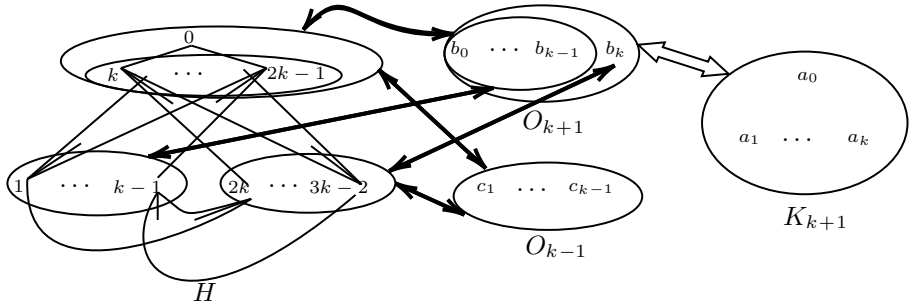

Note that $X=V(A) \cup V(H)$ is a switching set of $G$. We denote by $G'$ the graph obtained after applying GM switching in $G$ with respect to $X$. The main result of this section is as follows. 
\begin{theorem}\label{thm:GMregularcase}
For every $k\geq 2$, there exists a pair of $2k$-regular cospectral graphs with different zero forcing number.
\end{theorem}
\begin{proof}[Proof of Theorem \ref{thm:GMregularcase}]
We prove Theorem \ref{thm:GMregularcase} using several intermediate results, with the goal of showing that $Z(G) = 4k-2$ and $Z(G') \leq 4k-3$.

Let $k\geq 1$. We first define a bipartite graph $L_k$ with partite sets being $\{v_1,\ldots,v_k\}$ and $\{w_1,\ldots,w_k\}$ such that $v_i \sim w_j$ for every $i = 1,\ldots,k$ and $j = 1,\ldots,i$. The following lemma will be used later to find $Z(H)$.
\begin{lemma}\label{aboutL_k}
Let  $S$ be a set of blue vertices in $L_k$ with $|S|  =s $. Further, if there exists a vertex in $S$ with all of its neighbors in $S$ then we assume that $s\leq k-1$, otherwise $s\leq k-2$. Then any vertex forced by $S$ can not force further. 
\end{lemma}
\begin{proof}
We prove the assertion by induction on $k$. Clearly it holds for $k = 1,2$, and $3$. Let us assume it holds for all $k \leq n-1$ and consider $L_{n}$. Suppose there exists a vertex in $S$ with at most one white neighbor; otherwise, we are done. Without loss of generality, we assume that it is from $\{v_1,\ldots,v_{n}\}$, say $v_m$, where $m$ is the smallest such index. After the next force occurs, the vertices in $\{w_1,\ldots,w_m\}$ are all blue, but none of them can force at this point since all of them are connected to $\{v_{m+1},\ldots,v_{n}\}$ and at least two of these are white. The induced subgraph on $\{v_{m+1},\ldots,v_{n}\}$ and $\{w_{m+1},\ldots,w_{n}\}$ is isomorphic to $L_{n-m}$, and the zero forcing process in this subgraph is not affected by $\{v_1,\ldots,v_{m}\}$ and $\{w_1,\ldots,w_m\}$. Thus, by induction, the assertion follows. \qedwhite
\end{proof}

We now introduce some notation specific to the graph $H$ defined above. For a vertex $v\in V(H)$, let $N'(v) := N(v) \cup \{v\}$, where $N(v)$ is the set of vertices adjacent to $v$. For $i=0,\ldots, 3k-2$ and $\ell = 1,\ldots,k-1$, $\Gamma^+_{\ell}(i) := \{i+j\ (\text{mod } 3k-1)| j=1,\ldots,\ell\}$ and $\Gamma^-_{\ell}(i) := \{i-j\ (\text{mod } 3k-1) | j=1,\ldots,\ell\}$.  Note that for any $i$, $\{i\} \cup N(i) \cup \Gamma^+_{k-1}(i) \cup \Gamma^-_{k-1}(i) = V(H)$ with all four being independent sets in $H$. Moreover, the induced subgraph on $\Gamma^+_{k-1}(i) \cup \Gamma^{-}_{k-1}(i)$ is isomorphic to $L_{k-1}$. The following lemma is useful later on. 

\begin{lemma}\label{lemma_about_H}
Let $S$ be a set of blue vertices in $H$ with $|S| = k+\ell+s$, such that $\ell \leq k-3$ and $s=0$, $1$, or $2$, depending on whether there exists zero, one, or two vertices in $S$ with all of their neighbors in $S$. Then any vertex forced by $S$ can not force further. For any vertex $i\in S$ such that $|S \cap N(i)| \geq k-1$, none of the vertices in $S\setminus (i \cup \Gamma^{+}_{\ell}(i) \cup \Gamma^{-}_{\ell}(i))$ can perform a force. Moreover, $Z(H) = 2k-2$.
\end{lemma}

\begin{proof}
We assume that a vertex $i\in S$ exists such that $|S \cap N(i)| \geq k-1$; otherwise, no vertex can be forced in the next step. Since any vertex in $N(i)$ has exactly $k-1$ neighbors outside of $N'(i)$, none of them can make a force in the subsequent step. The first claim follows by Lemma \ref{aboutL_k}, and the observation that $N(v)$ is an independent set for all $v\in V(H)$. The second claim follows since any $i-j \in \Gamma_{k-1}^{-}(i)$ (or $i+j \in \Gamma_{k-1}^{+}(i)$) has exactly $j$ neighbors outside of $N'(i)$. Finally, any set of $2k-3$ blue vertices is not enough to be a zero forcing set for $H$. However, the set $S=\{0,2,3,\ldots,2k-2\}$ is a zero forcing set of order $2k-2$, since the vertex $0$ forces the vertex $2k-1$, the vertex $2k-1$ forces the vertex $1$, the vertex $1$ forces the vertex $2k$, the vertex $2$ forces the vertex $2k+1$ and so on until the vertex $k-1$ forces the vertex $3k-2$. Hence, $Z(H) = 2k-2$.
\qedwhite
\end{proof}

\begin{corollary}\label{cor_about_H}
Let $S$ be the same as in Lemma \ref{lemma_about_H}. Then after the final force the number of blue vertices in the set $\{0\} \cup N(0) \cup \Gamma_{k-1}^{-}(0)$ is at most $2k-2$.
\end{corollary}

The following lemma provides upper bounds on the zero forcing numbers of the cospectral graphs $G$ and $G'$. 

\begin{lemma}
Let $k\geq 2$. Then $Z(G) \leq 4k-2$ and $Z(G') \leq 4k-3$. 
\end{lemma}
\begin{proof}
Let $S = (V(A) \cup V(B) \cup V(C) \cup \{1,\ldots,k-1\})\setminus \{a_1,b_0\}$ be a set of blue vertices in $G$. Then the vertex $a_0$ forces the vertex $a_1$, the vertex $a_1$ forces the vertex $b_0$, vertices in $V(B)\setminus \{b_k\}$ force vertices in $\{0,k,\ldots,2k-2\}$, the vertex $0$ forces the vertex $2k-1$, and finally vertices in $\{1,\ldots,k-1\}$ force vertices in $\{2k,\ldots,3k-2\}$. Hence, $S$ is a zero forcing set of $G$. Since $|S| = 4k-2$, it follows that $Z(G) \leq 4k-2$. 

Let $S = V(A)  \cup V(C) \cup \{2,\ldots,k-1\} \cup \{2k,\ldots,3k-2\}$ be a set of blue vertices in $G'$. Then vertices in $V(A)$ force the vertices in $V(B)$, the vertex $c_1$ forces the vertex ${1}$, and finally vertices in $\{1,\ldots,k-1\}\cup \{2k,\ldots,3k-2\}$ force vertices in $\{0\} \cup \{k,\ldots,2k-1\}$. Hence, $S$ is a zero forcing set of $G'$. Since $|S'| = 4k-3$, it follows that $Z(G')\leq 4k-3$. 
\qedwhite
\end{proof}

Since we aim to prove that $Z(G) \neq Z(G')$, it is sufficient to show that $Z(G) = 4k-2$. For that, we need to prove that any set $S$ of blue vertices such that $|S| = 4k-3$ is not a zero forcing set of $G$. Suppose to the contrary that $S$ is a zero forcing set. Let $a = |S\cap V(A)|$, $b = |S\cap V(B)|$, $c = |S \cap V(C)|$, and $h = |S\cap V(H)|$, so that $a+b+c+h =4k-3$. Note that the following observations hold.

\begin{enumerate}[(i)]
\item Any vertex which has a neighbor in $A$ has exactly $k$ neighbors in $A$. Thus, $a \geq k-1$.
\item  Any vertex which has a neighbor in $C$ has exactly $k-1$ neighbors in $C$. Thus, $c \geq k-2$. 
\item Any vertex which has a neighbor in $V(H)$ has at least $k$ neighbors in $V(H)$. Thus, $h \geq k-1$.  
\end{enumerate}
Hence, in order to complete the proof of Theorem \ref{thm:GMregularcase}, we divide the remaining argument  in six different cases depending on $a \in \{k-1,k,k+1\}$ and $c \in \{k-2,k-1\}$. 

\begin{description}
\item[Case 1:] Suppose $a = k+1$, $c = k-1$, $b=k+1-r$, and $h = k+r-4$ where $3\leq r \leq k+1$. Since $h\leq 2k-3$, and $b\leq k-2$ thus, an initial force is not from $V(A)$ and $V(C)$. Hence, there are only four possibilities: a vertex in $V(H)$ makes an initial force within $H$, a vertex in $V(H)$ makes an initial force in $B$, a vertex in $\{b_0,b_1,\ldots,b_{k-1}\}$ makes an initial force in $H$, or the vertex $b_k$ makes an initial force in $H$. For any one of them, there are at most $h+b = 2k-3$ blue vertices in $V(H)$. According to Lemma \ref{lemma_about_H} 
and Corollary \ref{cor_about_H}, they are not enough to force $H$, and to induce any force from $C$. We need to argue that no force is induced from $V(A)$ at any step. Instead, we argue that there are at most an $r-3$ number of forces from $V(H)$ into $B$. For this, we use the first claim of Lemma \ref{lemma_about_H}. The $b$ number of vertices from $S\cap B$ may force vertices in $H$; those can further force in $H$ but cannot perform any subsequent force. For the first two possibilities, since $k$ vertices out of the $h$ vertices will not be able to force further, thus follows. Suppose the vertex $b_i$ makes an initial force in $H$ for the third possibility. Then we must have $|S \cap \{1,\ldots,k-1\}| \geq k-2$, and if the equality holds then the neighbor of $b_i$ in $\{0,k,\ldots,2k-2\}$ is in $S$. Since none of the $\{1,\ldots,k-1\}$ force further, and if the neighbor of $b_i$ was in $S$ it also can not force further, thus follows. For the fourth possibility, we must have $|S\cap \{2k-1,\ldots,3k-2\}| = k-1$. Since vertices in $\{2k-1,2k,\ldots,3k-2\}$ can only force $b_k$ in $B$, but that is already blue, thus follows.

\item[Case 2:] Suppose $a = k+1$, $c = k-2$, $b= k+1-r$, and $h = k+r-3$ where $2\leq r\leq k+1$. If $r=2$, then the blue vertices from $V(A)$ force the two white vertices in $B$. After that, there are only two possibilities: a vertex in $\{b_0,b_1,\ldots,b_{k-1}\}$ makes the next force in $H$ or vertex $b_k$ makes the next force in $H$. If the former occurs, then vertices in  $\{0,1,\ldots,2k-2\}$ are all blue. If the latter occurs, then vertices in $\{2k-1,2k,\ldots,3k-2\}$ are all blue. Since $C$ has one white vertex, there is no further force. 

Let $3\leq r \leq k+1$. Since $h \leq 2k-2$, and $b\leq k-2$ thus, an initial force is not from $V(A)$ and $V(C)$. Hence, there are only three possibilities: a vertex in $V(H)$ makes an initial force in $C$, a vertex in $\{b_0,b_1,\ldots,b_{k-1}\}$ makes an initial force in $H$, or the vertex $b_k$ makes an initial force in $H$. For any one of them, there are at most $2k-3+t$ blue vertices in $V(H)$ with $t=0,1$ number of vertices whose neighbors are all blue since $C$ has one white vertex. According to Lemma \ref{lemma_about_H} and Corollary \ref{cor_about_H}, they are not enough to force $H$, and to induce any force from $C$. We need to argue that no force is induced from $V(A)$ at any step. Instead, we argue that there are strictly less than an $r-2$ number of forces from $V(H)$ into $B$. For the first possibility, since $k+1$ vertices out of the $h$ vertices will not be able to force further, thus follows. Suppose the vertex $b_i$ makes an initial force in $H$ for the second possibility. Then we must have $|S \cap \{1,\ldots,k-1\}| \geq k-2$, and if the equality holds then the neighbor of $b_i$ in $\{0,k,\ldots,2k-2\}$ is in $S$. Since none of the $\{1,\ldots,k-1\}$ force further, and if the neighbor of $b_i$ was in $S$ it also can not force further, thus follows. For the fourth possibility, we must have $|S\cap \{2k-1,\ldots,3k-2\}| = k-1$. Since vertices in $\{2k-1,2k,\ldots,3k-2\}$ can only force $b_k$ in $B$, but that is already blue, thus follows.

\item[Case 3:] Suppose $a = k$, $c = k-1$, $b = k+1-r$, and $h = k+r-3$ where $2\leq r \leq k+1$. If $r=2$, then there are two possibilities: either all blue vertices from $\{b_0,\ldots,b_{k-1}\}$ force in $H$, or the vertex $b_k$ is blue and forces a vertex in $H$. For any of them, there is no further force. 

Let $3\leq r \leq k+1$. Since $h \leq 2k-2$, and $b\leq k-2$ thus, an initial force is not from $V(A)$ and $V(C)$. Hence, there are five possibilities: a vertex in $V(H)$ makes an initial force within $H$, a vertex in $V(H)$ makes an initial force in $B$, a vertex in $\{b_0,b_1,\ldots,b_{k-1}\}$ makes an initial force in $H$, the vertex $b_k$ makes an initial force in $H$, or a vertex in $V(B)$ makes an initial force in $A$. For any one of them, there are at most $2k-3+t$ blue vertices in $V(H)$ with $t=0,1$ number of vertices whose neighbors are all blue since $A$ has one white vertex. According to Lemma \ref{lemma_about_H} and Corollary \ref{cor_about_H}, they are not enough to force $H$, and to induce any force from $C$. We need to argue that no force is induced from $V(A)$ at any step. Instead, we argue that there are strictly less than an $r-2$ number of forces from $V(H)$ into $B$. For the first two possibilities, since $k$ vertices out of the $h$ vertices will not be able to force further, thus follows. Suppose the vertex $b_i$ makes an initial force in $H$ for the third possibility. Then we must have $|S \cap \{1,\ldots,k-1\}| \geq k-2$, and if the equality holds then the neighbor of $b_i$ in $\{0,k,\ldots,2k-2\}$ is in $S$. Since none of the $\{1,\ldots,k-1\}$ force further, and if the neighbor of $b_i$ was in $S$ it also can not force further, thus follows. For the fourth possibility, we must have $|S\cap \{2k-1,\ldots,3k-2\}| = k-1$. Since vertices in $\{2k-1,2k,\ldots,3k-2\}$ can only force $b_k$ in $B$, but that is already blue, thus follows. Finally, if we have an initial force from some vertex $b_i$ into $A$, we may have the next force from a vertex in $V(B)$ into $H$. In that case, we argue similarly as to the argument for the third and fourth possibilities above. 

\item[Case 4:] Suppose $a = k$, $c = k-2$, $b=k+1-r$, and $h = k+r-2$ where $1\leq r\leq k+1$. If $r=1$, then both white vertices in $V(A)$ and $V(B)$ will be forced into blue. After that there are only two possibilities: a vertex in $\{b_0,b_1,\ldots,b_{k-1}\}$ makes the next force in $H$, or the vertex $b_k$ makes the next force in $H$. If the former occurs, then vertices in $\{0,1,\ldots,2k-2\}$ are all blue. If the latter occurs, then vertices in $\{2k-1,2k,\ldots,3k-2\}$ are all blue. Since $C$ has one white vertex, there is no further force. If $r=2$, then there are only two possibilities: a vertex in $\{b_0,b_1,\ldots,b_{k-1}\}$ makes an initial force in $A$ or $H$, or the vertex $b_k$ makes an initial force in $A$ or $H$. If the former occurs, then vertices in $V(A)$, $V(B)$, and $\{0,1,\ldots,2k-2\}$ are all blue. If the latter occurs, then vertices in $V(A)$, $V(B)$, and $\{2k-1,2k,\ldots,3k-2\}$ are all blue. Since $C$ has one white vertex, there is no further force. If $r=k+1$, there is no further force after an initial force. 

Finally, let $3\leq r\leq k$. Since  $h \leq 2k-2$, and $b\leq k-2$ thus, an initial force is not from $V(A)$ and $V(C)$. Hence, there are only four possibilities: a vertex in $V(H)$ makes an initial force in $C$, a vertex in $\{b_0,b_1,\ldots,b_{k-1}\}$ makes an initial force in $H$, the vertex $b_k$ makes an initial force in $H$, or a vertex in $V(B)$ makes a force in $A$. For any one of them, there are at most $2k-3+t$ blue vertices in $V(H)$ with $t=0,1,2$ number of vertices whose neighbors are all blue since both $A$ and $C$ have one white vertex. According to Lemma \ref{lemma_about_H} and Corollary \ref{cor_about_H}, they are not enough to force $H$ and, to induce any force from $C$. We need to argue that no force is induced from $V(A)$ at any step. Instead, we argue that there are strictly less than an $r-2$ number of forces from $V(H)$ into $B$. For the first possibility,  since $k+1$ vertices out of the $h$ vertices will not be able to force further, thus follows. Suppose the vertex $b_i$ makes an initial force in $H$ for the second possibility. Then we must have $|S \cap \{1,\ldots,k-1\}| \geq k-2$, and if the equality holds then the neighbor of $b_i$ in $\{0,k,\ldots,2k-2\}$ is in $S$. All of the vertices in $\{1,\ldots,k-1\}$ and the neighbors of $b_i$ in $S$ can not force further. Moreover, the next force must be from $V(H)$ into $C$, but that would use at least one vertex out of the $h-(k-1)$, which can not force in $B$. Thus there are strictly less than $h- (k-1) - 1 = r-2$ number of forces from $V(H)$ into $B$. For the third possibility, we must have $|S\cap \{2k-1,\ldots,3k-2\}| = k-1$.  All the vertices in $\{2k-1,2k,\ldots,3k-2\}$ can only force $b_k$ in $B$, but that is already blue. Moreover, the next force must be from $V(H)$ into $C$, but that would use at least one vertex out of the $h-(k-1)$ vertices, which can not force in $B$. Thus there are strictly less than an $r-2 = h-(k-1)-1$ number of forces from $V(H)$ into $B$. Finally, if we have an initial force from some $b_i$ into $A$, we may have the next force from a vertex in $V(B)$ into $H$. In that case, we argue similarly as to the argument for the second and third possibilities above.

\item[Case 5:] Suppose $a = k-1$, $c = k-1$, $b = k+1-r$, and $h = k+r-2$ where $1\leq r\leq k+1$.  If $r=1$, there is no initial force. If $r=2$, there is no further force after an initial force. Let $r=k+1$. Then $H$ can not force itself with $2k-1$ blue vertices since each vertex has at least one white neighbor in $B$. On the other hand, the only possible force outside of $H$ is from $C$, which is not enough.

Finally, let $3\leq r \leq k$. Since $h \leq 2k-2$, and $b\leq k-2$ thus, an initial force is not from $V(A)$ and $V(C)$. Hence, there are only four possibilities: a vertex in $V(H)$ makes an initial force within $H$, a vertex in $V(H)$ makes an initial force in $B$, a vertex in $\{b_0,b_1,\ldots,b_{k-1}\}$ makes an initial force in $A$, or the vertex $b_k$ makes an initial force in $A$. For any one of them, there are at most $2k-3+t$ blue vertices in $V(H)$ with $t=0,1,2$ number of vertices whose neighbors are all blue since $A$ has two white vertices.  According to Lemma \ref{lemma_about_H} and Corollary \ref{cor_about_H}, they are not enough to force $H$, and to induce any force from $C$. We need to argue that no force is induced from $V(A)$ at any step. Instead, we argue that there are strictly less than an $r-2$ number of forces from $V(H)$ into $B$. For the first two possibilities, since $k$ vertices out of the $h$ vertices will not be able to force further, thus follows. Suppose the vertex $b_i$ makes an initial force in $A$ for the third possibility. At the next step, if some vertex $b_j \in \{b_0,\ldots,b_{i-1},b_{i+1},\ldots,b_{k-1}\}$ forces the white vertex in $A$, then the neighbor of $b_j$ from $\{0,k,\ldots,2k-2\}$ must be in $S$. Since none of the $k+1$ neighbors of $b_i$ and $b_j$ out of the $h$ vertices will be able to force in $B$, thus follows. For the fourth possibility, we must have $|S\cap \{2k-1,\ldots,3k-2\}| = k$. Since $h\leq 2k-2$ there is no further force. 

\item[Case 6:] Suppose $a = k-1$, $c = k-2$, $b = k+1-r$, and $h = k+r-1$ where $0\leq r \leq k+1$. If $r=0$, there is no initial force. If $r=1$, only one force is possible, from $V(B)$ into $V(A)$, and nothing after that. If $r=k+1$, only one force is possible, from $V(C)$ into $V(H)$, and nothing after that. Let $r=k$. If an initial force is from a vertex in $V(C)$, then after the final force, vertices in  $\{1,\ldots,k-1\}$ are always white. On the other hand, if an initial force is not from a vertex in $V(C)$, there are $2k-1$ blue vertices in $H$ with two blue vertices whose neighbors are all blue, and we know these are not enough to force $H$. Further, no force is induced from $V(A)$ at any step. 

Finally, let $2\leq r \leq k-1$. Since $h \leq 2k-2$, $b\leq k-1$, and $a=k-1$ thus, an initial force is not from $V(A)$ and $V(C)$. Hence, there are only three possibilities: a vertex in $V(H)$ makes an initial force in $C$, a vertex in $\{b_0,b_1,\ldots,b_{k-1}\}$ makes an initial force in $A$, or the vertex $b_k$ makes an initial force in $A$. For any one of them, there are at most $2k-3+t$ blue vertices in $V(H)$ with $t=0,1,2$ number of vertices whose neighbors are all blue since $A$ has two white vertices and $C$ has one white vertex. According to Lemma \ref{lemma_about_H} and Corollary \ref{cor_about_H}, they are not enough to force $H$, and to induce any force from $C$. We need to argue that no force is induced from $V(A)$ at any step. Instead, we argue that there are strictly less than an $r-2$ number of forces from $V(H)$ into $B$. For the first possibility, since $k+1$ vertices out of the $h$ vertices will not be able to force further, thus follows. Suppose the vertex $b_i$ makes an initial force in $A$ for the second possibility. At the next step, if some vertex $b_j \in \{b_0,\ldots,b_{i-1},b_{i+1},\ldots,b_{k-1}\}$ forces the white vertex in $A$, then the neighbor of $b_j$ from $\{0,k,\ldots,2k-2\}$ must be in $S$. Since none of the $k+1$ neighbors of $b_i$ and $b_j$ out of the $h$ vertices will be able to force in $B$, thus follows. For the third possibility, we must have $|S\cap \{2k-1,\ldots,3k-2\}| = k$. Since $h\leq 2k-2$ there is no further force. 
\end{description}

This completes the proof of Theorem \ref{thm:GMregularcase}. 
\qedwhite
\end{proof}

\section*{Acknowledgements}

The research of all the authors was partially supported by NSF grant 1916439. This project was started at the MRC workshop ``Finding Needles in Haystacks: Approaches to Inverse Problems using Combinatorics and Linear Algebra'', which took place in June 2021 with support from the National Science Foundation and the American Mathematical Society. The authors
are grateful to the organizers of this meeting. We thank S. Fallat for bringing reference \cite{fallatstudentphd} to our attention. \\A. Abiad is partially supported by the FWO (Research Foundation Flanders) grant 1285921N. The work of J.~Breen was supported in part by NSERC Discovery Grant RGPIN-2021-03775.


\end{document}